\documentclass[10pt]{article}
\usepackage{amsmath}
\usepackage[arrow, matrix, curve]{xy}
\usepackage{amssymb}
\usepackage{amsthm}
\newtheorem{theorem}{Theorem}[section]
\newtheorem{lemma}[theorem]{Lemma}
\newtheorem{corollary}[theorem]{Corollary}
\newtheorem{proposition}[theorem]{Proposition}
\theoremstyle{definition}

\theoremstyle{remark}

\numberwithin{equation}{section}

\DeclareMathOperator{\su}{SU}
\DeclareMathOperator{\psu}{PSU}
\DeclareMathOperator{\Hom}{Hom}
\DeclareMathOperator{\nr}{nr}
\DeclareMathOperator{\vol}{vol}
\DeclareMathOperator{\sign}{sign}
\DeclareMathOperator{\kod}{kod}
\newcommand{\CC}{\mathbb C}
\newcommand{\BB}{\mathbb B}
\newcommand{\DD}{\mathbb D}
\newcommand{\PP}{\mathbb P}
\newcommand{\QQ}{\mathbb Q}
\newcommand{\G}{\mathbf G}
\newcommand{\e}{\mathbf e}
\newcommand{\ttt}{\mathbf{sign}}
\newcommand{\A}{{\mathbf A}}
\newcommand{\oo}{\mathfrak o}
\DeclareMathOperator{\gal}{Gal}
\newcommand{\OO}{\mathcal O}
\begin{document}
\title{Arithmetic of a fake projective plane and related elliptic surfaces}
\author{Amir D\v{z}ambi\'c\\ \\ Institut f\"ur Mathematik\\ Johann Wolfgang Goethe-Universit\"at\\ 
Robert-Mayer Str.~6-8, 60325 Frankfurt am Main\\ Email: dzambic@math.uni-frankfurt.de}
\date{}
\maketitle
\begin{abstract}
The purpose of the present paper is to explain the fake projective plane constructed by J.~H.~Keum from the point of view of arithmetic ball quotients. Beside the ball quotient associated with the fake projective plane, we also analize two further naturally related ball quotients whose minimal desingularizations lead to two elliptic surfaces, one already considered by J.~H.~Keum as well as the one constructed by M.~N.~Ishida in terms of p-adic uniformization. 
\end{abstract}

\section{Introduction}
In 1954, F.~Severi raised the question if every smooth complex algebraic surface homeomorphic to the projective plane $\PP_2(\CC)$ is also isomorphic to $\PP_2(\CC)$ as an algebraic variety. To that point, this was classically known to be true in dimension one, being equivalent to the statement that every compact Riemann surface of genus zero is isomorphic to $\PP_1(\CC)$. F.~Hirzebruch and K.~Kodaira were able to show that in all odd dimensions $\PP_n(\CC)$ is the only algebraic manifold in its homeomorphism class. But it took over 20 years until Severi's question could be positively answered. One obtains it as a consequence of S-T.~Yau's famous results on the existence of K\"ahler-Einstein metrics on complex manifolds. Two years after Yau's results, in \cite{mum2}, D.~Mumford discussed the question, if there could exist algebraic surfaces which are not isomorphic to $\PP_2(\CC)$, but which are topologically close to $\PP_2(\CC)$, in the sense that they have same Betti numbers as $\PP_2(\CC)$. Such surfaces are nowadays commonly called {\em fake projective planes}, see \cite{bhpv}. The following characterization of fake projective planes follows immediately from standard results in the theory of algebraic surfaces in combination with above mentioned Yau's result:
\begin{lemma}
\label{fpp}
A smooth algebraic surface $X$ is a fake projective plane if and only if $c_2(X)=3$, $c_1^2(X)=9$, $q(X)=p_g(X)=0$, and $\kod(X)=2$. In particular, the universal covering of $X$ is isomorphic to the unit ball $\BB_2\subset \CC^2$ and consequently 
\begin{equation}
\label{1}
X\cong \Gamma\backslash \BB_2
\end{equation}
where $\Gamma$ is a discrete, cocompact, and torsion free subgroup of $Aut(\BB_2)\cong \psu(2,1)$.  
\end{lemma}
Here, $c_2(X)$ and $c_1^2(X)$ denote the two Chern numbers of $X$ which are interpreted as the Euler number and the selfintersection number of the canonical divisor respectively, $q(X)$ is the irregularity, $p_g(X)$ the geometric genus, and $\kod(X)$ is the Kodaira dimension of $X$.\\

In the above mentioned work \cite{mum2}, Mumford was also able to show the existence of fake projective planes, constructing an example. However, his construction is based on the theory of p-adic uniformization and his example is not presented in the form (\ref{1}), as one naturally would expect. Moreover, his example is not even a complex surface, but a surface defined over the field of 2-adic numbers $\overline{\mathbb Q}{_2}$. But, p-adic methods were for long time the only way for producing examples of fake projective planes, of which only fnitely many can exist, as pointed out by Mumford. Further examples of p-adic nature have been given by M.~-N.~Ishida and F.~Kato (\cite{Ishida-Kato}), whereas the first complex geometric example seems to be the one constructed by J.~H.~Keum in \cite{keum}. Motivated by the work of M.~N.~Ishida (\cite{Ishida}), the author finds a fake projective plane as a degree 7 (ramified) cyclic covering of an explicitely given properly elliptic surface. Again, as all the examples before, Keum's example is not given as a ball quotient. The breakthrough in the study of fake projective planes came with the recent work of G.~Prasad and S.~Yeung, \cite{PS}, where the authors succeeded to determine all fake projective planes. The main technical tool in their proof is a general volume formula developed by Prasad which is applied to the case of $\su(2,1)$, and combined with the fact that the fundamental group of a fake projective plane is arithmetic. The resulting arithmetic groups are given rather explicitely in terms of Bruhat-Tits theory.\\ 

In the following paper we identify Keum's fake projective plane with a ball quotient $X_{\Gamma^{'} }=\Gamma^{'}\backslash\BB_2$. In fact, this ball quotient appears in \cite{PS} (see \cite{PS}, 5.~9, and there the examples associated with the pair (7,2)). However, in this paper we use a slightly modified approach to this quotient, motivated by \cite{Kat}, who identified Mumford's fake projective plane as a connected component of a certain Shimura variety. Moreover, Mumford's 2-adic example can be considered as a kind of a ``2-adic completion`` of a ball quotient. This ball quotient also appears in \cite{PS} and is also associated to the pair $(7,2)$ (in the sense of \cite{PS}), but this ball quotient is not isomorphic to $X_{\Gamma^{'}}$ since it doesn't admit an automorphism of order $7$.\\ 
Let us briefly describe the approach. We start with an explicit division algebra $D$ over $\QQ$ with an involution of second kind $\iota_b$, a particular maximal order $\OO$, and we consider the arithmetic group $\Gamma=\Gamma_{\OO,b}$ consisting of all norm-1 elements in $\OO$ which are unitary with respect to the hermitian form corresponding to $\iota_b$. Now,  $\Gamma^{'}$ appears as a principal congruence subgroup of index $7$ in $\Gamma$. The explicit knowledge of $\Gamma$ allows us to see particular elements of finite order in $\Gamma$ and gives us the possibility to explain the elliptic surface appearing in \cite{keum} from the point of view of ball quotients, namely as the minimal desingularization of quotient singularities of $X_{\Gamma}=\Gamma\backslash \BB_2$. Passing to a particular group $\tilde\Gamma$ containing $\Gamma$ with index 3, we identify the minimal desingularization of the ball quotient $X_{\tilde\Gamma}$ with the elliptic surface of Ishida (\cite{Ishida}) which is originally given in terms of p-adic uniformization. We illustrate the situation in the following diagram:
\begin{displaymath}
\begin{xy}
 \xymatrix{
 & X_{\Gamma^{'}}\ar[d]_7\\
\widetilde{X_{\Gamma}}\ar[d]\ar[r] & X_{\Gamma}\ar[d]_3\\
\widetilde{X_{\tilde\Gamma}}\ar[r] & X_{\tilde\Gamma}
}
\end{xy}
\end{displaymath}
There, the arrows indicate finite cyclic coverings of compact ball quotients with announced degree, $X_{\Gamma^{'}}$ is a fake projective plane, $X_{\Gamma}$ and  $X_{\tilde\Gamma}$ are singular ball quotients, having only cyclic singularities and $\widetilde{X_{\Gamma}}, \widetilde{X_{\tilde\Gamma}}$ are the canonical resolutions of singularities and are both smooth minimal elliptic surfaces of Kodaira dimension one. Identifying $\widetilde{X_{\tilde\Gamma}}$ with Ishida's elliptic surface in \cite{Ishida}, we know the singular fibers of its elliptic fibration. Explicit knowledge of the finite covering $X_{\Gamma}\longrightarrow X_{\tilde\Gamma}$ gives the elliptic fibration of $\widetilde{X_{\Gamma}}$, already determined by Keum.

\section{Preliminiaries on arithmetic ball quotients}
In this section, we discuss arithmetically defined groups which act properly discontinuously on a symmetric domain isomorphic to the two-dimensional complex unit ball and collect some basic properties of the corresponding locally symmetric spaces.
\subsection{Arithmetic lattices}
If $H$ is a hermitian form over $\CC$ in three variables with two negative and one positive eigenvalue, then we speak of a form with signature $(2,1)$. The set of positive definite lines
\begin{equation}
 \BB_H=\{[l]\in \PP_2(\CC)\mid H(l,l)> 0\}\subset \PP_2(\CC)
\end{equation}
with respect to such a hermitian form $H$ is isomorphic to the two dimensional complex unit ball $\BB_2$. Alternatively, we can see $\BB_H$ as the symmetric space $\BB_H\cong \su(H)/K_{0}$ associated with the Lie group $\su(H)$, that is the group of isometries with respect to H of determinant 1, where $K_{0}$ is a maximal compact subgroup in $\su(H)$. Every cocompact discrete and torsion free subgroup $\Gamma$ of $\su(H)$ acts properly discontinously on $\BB_H$ as a group of linear fractional transformations, but not effectively in general. However, the image $\PP\Gamma$ of $\Gamma$ in $\psu(H)$ acts effectively. The orbit space $X_{\Gamma}= \Gamma\backslash \BB_H$ has a natural structure of a complex manifold, and even more: it has the structure of a smooth projective algebraic variety. Arithmetic subgroups of $\su(H)$ provide a large natural class of discrete groups which act on $\BB_H$. By the classification theory of forms of algebraic groups, all arithmetic groups which act on the ball can be constructed as follows:\\
Let $F$ be a totally real number field and $K/F$ a pure imaginary quadratic extension (CM extension) of $K$. Let $A$ be a 9-dimensional central simple algebra over $K$ and assume that on $A$ exists an {\em involution of second kind}, i.~e.~an anti-automorphism $\iota:A\rightarrow A$ such that $\iota^{2}=id$ and the restriction $\iota|_{K}$ is the complex conjugation $x\mapsto \overline{x}\in \gal(K/F)$. In that case, using the Skolem-Noether theorem, we can always normalize $\iota$ in such a way that the extension $\iota_{\CC}$ on $A\otimes \CC\cong M_3(\CC)$ of $\iota$ is the hermitian conjugation, $\iota_{\CC}(m)=\overline{m}^t$. In this case we say that $\iota$ is the {\em canonical involution of second kind}.\\ 
As a central simple algebra over a number field, $A$ is a cyclic algebra 
\begin{equation}
\label{cyclalg}
A= A(L,\sigma,\alpha)= L\oplus Lu\oplus Lu^2,
\end{equation}
where $L/K$ is an (cyclic) extension of number fields of degree 3, $\sigma$ is a generator of $\gal(L/K)$ and $u\in A$ satisfies $\alpha=u^3\in K^{\ast}$, $a u=u a^{\sigma}$ for all $a\in L$. This data already determine the isomorphy class of $A$. The structure of a division algebra is determined by the class of $\alpha$ in $K^{\ast}/N_{L/K}(L^{\ast})$ by class field theory: $A$ is a division algebra if and only if $\alpha\notin N_{L/K}(L^{\ast})$, otherwise $A$ is the matrix algebra $M_3(K)$. We note that $L$ is a splitting field of $A$, i.~e.~$A\otimes L\cong M_3(L)$ and that we can embedd $A$ in $M_3(L)$ if we put:
\begin{equation}
\label{matrixrep}
a \mapsto \begin{pmatrix} a & 0 & 0\\ 0 & a^{\sigma} & 0\\ 0 & 0 & a^{\sigma\sigma} \end{pmatrix}\ \textrm{for}\ a\in L,\  u\mapsto \begin{pmatrix} 0 & 0 & \alpha\\ 1 & 0 & 0\\ 0 & 1 & 0 \end{pmatrix}
\end{equation}
and extend linearly to all $A$.\\ 
Consider again the canonical involution $\iota$ of second kind on $A$ and let $b\in A$ be an $\iota$-invariant element, i.~e.~an element with $b^{\iota}=b$. Then $\iota_b: a\mapsto b a^{\iota}b^{-1}$ defines a further involution of second kind. Let $\A^{(1)}$ denote the group of elements in $A$ of reduced norm $1$ considered as an algebraic group and let $\G_{b}=\{g\in\A^{(1)}\mid  gg^{\iota_b}=1\}$. Then, $\G_b$ is an algebraic group defined over $F$. Let us further assume that the matrix corresponding to $b$, obtained from the embedding $A\hookrightarrow M_3(\CC)$ induced by $id\in \Hom(F,\CC)$, represents a hermitian form of signature $(2,1)$, and for every $id\neq \tau\in \Hom(F,\CC)$ the induced matrix is a hermitian form of signature $(3,0)$. Then the group of real valued points $\G_b(\mathbb R)$ is isomorphic to the product $\su(2,1)\times \su(3)^{[F:\QQ]-1}$. Since $\su(3)$ is compact, according to the theorem of Borel and Harish-Chandra, every arithmetic subgroup of $\G_b(F)$ is a lattice in $\su(2,1)$, i.~e.~a discrete subgroup of finite covolume and acts properly discontinuously on the ball. The arithmetic subgroups derived from the pair $(A,\iota_b)$ can be specified in terms of orders in $A$: Every such group is commensurable to a group
\begin{displaymath}
 \Gamma_{\OO,b}=\{\gamma\in \OO\mid \gamma\gamma^{\iota_b}=1,\ \ \nr(\gamma)=1\},
\end{displaymath}
where $\OO$ is a $\iota_b$-invariant order in $A$ and $\nr(\cdot)$ denotes the reduced norm. For instance, take $A=M_3(K)$ and let $H\in M_3(K)$ be hermitian with the property that its signature is $(2,1)$ when considered as matrix over $\CC$ and that the signature of all matrices obtained by applying non-trivial Galois automorphisms $\tau\in \gal(F/\QQ)$ to the entries is $(3,0)$. $M_3(\oo_K)$ is definitively an order in $M_3(K)$ and the arithmetic group $\Gamma_{H}=\su(H,\oo_K)$ is called the (full) {\em Picard modular group}. On the other hand, the arithmetic lattices constructed from the division algebras are generally called {\em arithmetic lattices of second kind}.  

\subsection{Invariants of arithmetic ball quotients}
\label{inv}
Keeping the notations from the last paragraph, let $\G_b$ be an algebraic group derived from a pair $(A,\iota_b)$ for which $b$ satisfies the additional condition $\G_b(\mathbb R)\cong \su(2,1)\times \su(3)^{[F:\QQ]-1}$. Let $\Gamma$ be an arithmetic subgroup in $\G_b(F)$ and denote $X_{\Gamma}=\Gamma\backslash \BB_2$ the corresponding locally symmetric space. Then, the Godement's compactness criterion implies that $X_{\Gamma}$ is compact, except in the case where $A$ is the matrix algebra over an imaginary quadratic field $K$. After a possible descent to a finite index normal subgroup, we can assume that $\Gamma$ is torsion free and $X_{\Gamma}$ is smooth. There is always a volume form $\mu$ on $\BB_2$ such that the volume $\vol_{\mu}(\Gamma)$ of a fundamental domain of $\Gamma$ is exactly the Euler number of $X_{\Gamma}$, when $\Gamma$ is torsion free and cocompact. Under the assumption that the arithmetic group is so-called {\em principal arithmetic subgroup} this volume can be given explicitely by formulas involving exclusively data of arithmetical nature. A principal arithmetic group $\Lambda$ is defined as $\Lambda=\G_b(F)\cap \prod_{v} P_v$, where $\{P_v\}$ is a collection of parahoric subgroups $P_v\subset \G_b(F_v)$ ($v$ a non-archimedian place of $F$), such that $\prod_{v} P_v$ is open in the adelic group $\G_b(\mathbb A_F)$ (see \cite{pras}, 3.~4, or \cite{prasbo},1.~4. for details). Let us recall this formula for principal arithmetic subgroups of $\su(2,1)$ established in \cite{PS} where the reader will find omitted details (see also \cite{pras} and \cite{prasbo} for the general case). Let $D_K$ and $D_F$ denote the discriminants of the number fields $K$ and $F$ and $\zeta_F(\cdot)$ the Dedekind zeta function of $F$. For $Re(s)>1$ a $L$-function is defined by $L(s,\chi_{K/F})=\prod_{v}(1-\chi_{K/F}(\mathfrak p_v)N(\mathfrak p_v)^{-s})^{-1}$ where $v$ runs over all finite places of $F$, $\mathfrak p_v$ denotes the prime ideal of $\oo_F$ corresponding to $v$, $N(\mathfrak p_v)=|\oo_F/\mathfrak p_v|$ and $\chi_{K/F}(\cdot)$ is a character defined to be 1,-1 or 0 according to whether $\mathfrak p_v$ splits, remains prime or ramifies in $K$.
\begin{lemma}[see \cite{PS}]
\label{vol}
 Let  $\Lambda\subset\G_b(F)$ be a principal arithmetic subgroup. Then
\begin{displaymath}
 \vol_{\mu}(\Lambda)= 3\frac{D_K^{5/2}}{D_F}(16\pi^5)^{-[F:\QQ]}\zeta_F(2)L(3,\chi_{K/F})\mathcal E
\end{displaymath}
where $\mathcal E=\prod_{v\in S} e(v)$ is a product running over a finite set $S$ of non-archimedian places of $F$ determined by the localization of $\Lambda$ with rational numbers $e(v)$, given explicitely in \cite{PS} 2.~5. 
 \end{lemma}
The above formula not only gives the Euler number of a smooth ball quotient $X_{\Gamma}$, when $\Gamma$ is torsion free finite index normal subgroup of a principal arithmetic $\Lambda$, but also other numerical invariants. Namely, by Hirzebruch's proportionality theorem, $c_1^2(X_{\Gamma})=3c_2(X_{\Gamma})$ for any smooth and compact ball quotient. Consequently, the Noether formula implies for the Euler-Poincar\'e characteristic $\chi(X_{\Gamma}):=\chi(\OO_{X_{\Gamma}})$ of the structure sheaf $\OO_{X_{\Gamma}}$ (arithmetic genus): $\chi(X_{\Gamma})=c_2(X_{\Gamma})/3$. Similarly, the signature $\sign(X_{\Gamma})$ equals to $c_2(X_{\Gamma})/3$ by Hirzebruch's signature theorem. In general, the remaining Hodge numbers (irregularity and the geometric genus) are not immediately given. But, for a large class of arithmetic groups, namely congruence subgroups of second kind, i.~e.~those defined by congruences and contained in division algebras, there is a vanishing theorem of Rogawski (see \cite{br}, theorem 1), saying that for such groups $H^1(\Gamma,\CC)$ vanishes. Then it follows that the irregularity of the corresponding ball quotients vanishes, since we can identify the two cohomology groups $H^{\ast}(\Gamma)$ and $H^{\ast}(X_{\Gamma})$.
\section{Construction of the fake projective plane}
\label{constr}
Let $\zeta=\zeta_7=\exp(2\pi i/7)$ and $L=\QQ(\zeta)$. Then, $L$ contains the quadratic subfield $K=\QQ(\lambda)\cong\QQ(\sqrt{-7})$ with $\lambda=\zeta+\zeta^2+\zeta^4=\frac{-1+\sqrt{-7}}{2}$. The automorphism $\sigma: \zeta\mapsto \zeta^2$ generates a subgroup of $\gal(L/\QQ)$ of index 2 and leaves $K$ invariant, therefore, $\langle\sigma\rangle=\gal(L/K)$. We put $\alpha=\lambda/\overline{\lambda}$. As we have seen before (compare (\ref{cyclalg})), the triple $(L,\sigma,\alpha)$ defines a cyclic algebra $D=D(L,\sigma,\alpha)$ over $K$. 
\begin{lemma}
The algebra $D$ is a division algebra and has an involution of second kind. The assignement $a\mapsto \bar a$ for $a\in L$, $u\mapsto \bar\alpha u^2$ defines the canonical involution of second kind $\iota$. Let $b=tr(\lambda)+\bar\lambda u+\bar\lambda u^2$. Then, the induced hermitian matrix $H_b$ has the signature $(2,1)$.
\end{lemma}
\begin{proof}
The choice of $\alpha$ ensures that $\alpha\notin N_{L/K}(L^{\ast})$ by Hilbert's theorem 90. This proves the first statement. The remaining statements are proven in an elementary way, using the matrix representation of $D$ given in (\ref{matrixrep}). 
\end{proof}
Hence, the algebraic group $\G_b$ is a $\QQ$-form of the real group $\su(2,1)$. Now, we construct an arithmetic subgroup in $\G_b(\QQ)$ derived from a maximal order in $D$. For this let 
\begin{equation}
 \OO=\oo_L\oplus\oo_L\bar\lambda u\oplus \oo_L\bar\lambda u^2.
\end{equation}
Clearly, $\OO$ is an order in $D$. Also one easily sees that $\OO$ is invariant under the involution $\iota_b$ defined by $b$. We know even more:
\begin{lemma}
$\OO$ is a maximal order in $D$.
\end{lemma}
\begin{proof}
$\OO$ is maximal if and only if the localization $\OO_v$ is maximal for every finite place $v$. Over a local field, any central simple algebra $A_v$ contains (up to a conjugation) the unique maximal order $\mathcal M_v$ (see \cite{reiner}). Therefore the discriminant $d(\mathcal M_v)$ completely characterizes $\mathcal M_v$. The discriminant $d(\OO)$ is easily computed to be $2^6$. Since $d(\OO_v)=d(\OO)_v$, we immediately see that at all places $v$ not dividing 2, $d(\OO_v)=1$. Exactly at those places $D_v$ is the matrix algebra, since $\alpha$ is an unit there, and $\OO_v$ is maximal by \cite{reiner}, p.~185. At the two places $\lambda$ and $\bar\lambda$ dividing 2, $D_v$ is a division algebra. There $d(\OO_v)$ is exactly, the discriminant of the maximal order $\mathcal M_v$ (\cite{reiner}, p.~151).
\end{proof}
Let 
\begin{equation}
\Gamma_{\OO,b}=\G_b(\QQ)\cap\OO=\{\gamma\in \OO\mid \gamma\gamma^{\iota_b}=1,\ \nr(\gamma)=1\}
\end{equation}
be the arithmetic subgroup of $\G_b(\QQ)$ defined by $\OO$. We shall summarize some properties of $\Gamma_{\OO,b}$:
\begin{lemma}
\label{7}
$\Gamma_{\OO,b}$ is a principal arithmetic subgroup. Every torsion element in $\Gamma_{\OO,b}$ has the order 7. All such elements are conjugate in $D$. 
\end{lemma}

\begin{proof}
By definition, $\Gamma_{\OO,b}$ will be principal if at all finite places $p$ of $\QQ$ its localization is a parahoric subgroup of $\G_b(\QQ_p)$. Since $\OO$ is maximal, at all places $p\neq 2$  the localization $\Gamma^{[p]}_{\OO,b}$ is the special unitary group $\su(H_b, \oo_p)$, where $\oo_p=\oo_K\otimes \QQ_p$. Then by \cite{Tits}, $\Gamma^{[p]}_{\OO,b}$ is maximal parahoric. Since 2 is split in $K$, there is a division algebra $D_2$ over $\QQ_2$ such that $D\otimes\QQ_2=D_2\oplus D_2^{o}$, where $D_2^{o}$ denotes the opposite algebra to $D_2$. The projection to the first factor gives an isomorphism $\G_b(\QQ_2)\cong D^{(1)}_{2}$, the group of elements of reduced norm 1 in $D_2$. Let $\mathcal M_2$ be the maximal order in $D_2$. Then $\Gamma^{[2]}_{\OO,b}=$ $\mathcal M^{(1)}_2$.  Again, by \cite{Tits}, this is a maximal parahoric group. In order to prove the second statement let us consider an element $\tau$ of finite order in $\Gamma_{\OO,b}$. Let $\eta$ be an eigenvalue of $\tau$. Then $\eta$ is a root of unity and $\QQ(\eta)$ is a commutative subfield of $D$. Conversly, every cyclotomic subfield of $D$ containing $K$ gives rise to an element of finite order in $D$. Consequently, we have a bijection between the set of the conjugacy classes of elements of finite order in $D$ and the cyclotomic fields $C\subset D$ which contain the center $K$ of $D$. Since $L$ is the only such field, only elements of order 7, 2 and 14 can occur. But since the reduced norm of $-1$ is $-1$ again, elements of order $2$ don't belong to $\Gamma$. Thus, only elements of order $7$ are possible. 
\end{proof}
Let us now consider a particular congruence subgroup of $\Gamma_{\OO,b}$, namely the principal congruence subgroup
\begin{equation}
\Gamma_{\OO,b}(\lambda)=\{\gamma\in \Gamma_{\OO,b}\mid \gamma\equiv 1\bmod \lambda\}
\end{equation} 
We have
\begin{lemma}
$\Gamma_{\OO,b}(\lambda)$ is torsion free subgroup of index $[\Gamma_{\OO,b}:\Gamma_{\OO,b}(\lambda)]$=7. 
\end{lemma}
\begin{proof}
By lemma \ref{7} we have to show that $\Gamma_{\OO,b}(\lambda)$ contains no elements of order 7. Let $\gamma$ be an element in $\Gamma_{\OO,b}(\lambda)$ of finite order $k$. The eigenvalues of the representing matrix $m_{\gamma}$ of $\gamma$ are $k$-th roots of unity. Let $\eta$ be such an eigenvalue and $E=\QQ(\eta)$. Then $E\cong L$ by arguments in the proof of lemma \ref{7}. Since $\gamma$ belongs to the congruence subgroup defined by $\lambda$, $\lambda$ divides the coefficients of $m_{\gamma}-1_3\in M_3(E)$. Let $x$ be an eigenvector of $m_{\gamma}$. Multiplying with an integer we can assume $x\in \oo_E^3$. Then $\lambda|(m_{\gamma}-1_3)x=(\eta-1)x$ from which follows that $\lambda$ divides $\eta-1$ in $\oo_E$. Taking the norms we have $N_{E/\QQ}(\lambda)|N_{E/\QQ}(\eta-1)|k$. This is not possible when assuming $k=7$. Therefore, $\Gamma_{\OO,b}(\lambda)$ is torsion free. In order to compute the index, we make use of the strong approximation property which holds for $\G_b$. It allows us to express the index $[\Gamma_{\OO,b}:\Gamma_{\OO,b}(\mathfrak a)]$ of an arbitrary principal congruence subgroup $\Gamma_{\OO,b}(\mathfrak a)$ defined by some ideal $\mathfrak a=\prod \mathfrak p^{n_{\mathfrak p}}$ of $\oo_K$ as a product of local indices $\prod_{p|\mathfrak a}[\Gamma^{[p]}_{\OO,b}:\Gamma_{\OO,b}^{[p]}(\mathfrak p^{n_{\mathfrak p}})]$, where $p=\mathfrak p\cap\QQ$. In the case in question, we have $[\Gamma_{\OO,b}:\Gamma_{\OO,b}(\lambda)]=[\Gamma^{[2]}_{\OO,b}:\Gamma^{[2]}_{\OO,b}(\lambda)]$. But in the proof of lemma \ref{7} we already determined the structure of the localizations of $\Gamma_{\OO,b}$: $\Gamma^{[2]}_{\OO,b}\cong \mathcal M_2^{(1)}$ and therefore $\Gamma^{[2]}_{\OO,b}(\lambda)$ is the congruence subgroup  $\mathcal M_2^{(1)}(\pi_{D_2})$, where $\pi_{D_2}$ is the uniformizing element of $D_2$. It follows from a theorem of Riehm (\cite{Rie} Theorem 7, see also \cite{PS}) that $[\mathcal M^{(1)}_2:\mathcal M^{(1)}_2(\pi_{D_2})]=[\mathbb F^{\ast}_{2^3}:\mathbb F^{\ast}_2]=7$. 
\end{proof}
Let us in the following shortly write $\Gamma$ for $\Gamma_{\OO,b}$ and $\Gamma^{'}$ for $\Gamma_{\OO,b}(\lambda)$. 
The main result of this section is
\begin{theorem}
The ball quotient $X_{\Gamma^{'}}$ is a fake projective plane.
\end{theorem}
\begin{proof}
First we would like to compute the Euler number $c_2(X_{\Gamma^{'}})$ of $X_{\Gamma^{'}}$. Since $\Gamma^{'}$ is torsion free, $c_2(X_{\Gamma^{'}})=vol(\Gamma^{'})=[\Gamma:\Gamma^{'}]vol_{\mu}(\Gamma)=7vol_{\mu}(\Gamma)$. By lemma \ref{7} $\Gamma$ is principal, so we can apply lemma \ref{vol} in order to compute $vol_{\mu}(\Gamma)$. Well known is the value $\zeta_{\QQ}(2)=\pi^2/6$. The other value $L(3,\chi_K)=-\frac{7}{8}\pi^3 7^{-5/2}$ is computed using functional equation and the explicit formula for generalized Bernoulli numbers. In the last step, we determine the local factors $\mathcal E$. Looking at \cite{PS}, 2.~2.~non trivial local factors $e(v)$ can only occur for $v=2$ and $v=7$. Sections 2.~4.~and 2.~5.~of \cite{PS} give $e(2)=3$ and $e(7)=1$ since the localizations of $\Gamma$ are maximal parahoric.  Altoghether we get $vol_{\mu}(\Gamma)=3/7$ and $c_2(X_{\Gamma^{'}})=3$. Proportionality theorem gives $c_1^2(X_{\Gamma^{'}})=9$. Rogawski's vanishing result implies $q(X_{\Gamma^{'}})=0$. Then automatically $p_g(X_{\Gamma^{'}})=0$. As a smooth compact ball quotient $X_{\Gamma^{'}}$ is a surface of general type. By lemma \ref{fpp} $X_{\Gamma^{'}}$ is a fake projective plane. 
\end{proof}

\section{Structure of $X_{\Gamma}$}
\label{struct}
Let the notations be as in the last section and in particular $\Gamma:=\Gamma_{\OO,b}$, $\Gamma^{'}:=\Gamma_{\OO,b}(\lambda)$, let in addition $\BB$ denote the ball defined by (the matrix representation of) $b$. In this section we are interested in the structure of the ball quotient $X_{\Gamma}=\Gamma\backslash \BB$ by the arithmetic group $\Gamma$. According to lemma \ref{7}, the elements of finite order in $D$ correspond to the $7$-th roots of unity. Hence, all elements of finite order in $\Gamma$ are conjugated to a power of $\zeta=\zeta_7$.  The torsion element $\zeta$ doesn't belong to $\Gamma$, since $b$ is not invariant under the operation $b\mapsto \zeta b\zeta^{\iota}$. But $\zeta c\zeta^{\iota}=c$ for $c=\zeta+\zeta^{-1}$, which is $\iota$-invariant element of signature $(2,1)$. For this reason $Z=g^{-1}\zeta g$ is an element of order $7$ in $\Gamma$, where $g\in D$ is chosen such that $g b g^{-1}=c$. Therefore $X_{\Gamma}$ is isomorphic to the quotient $X_{\Gamma^{'}}/\langle Z\rangle$ by the finite subgroup $\langle Z\rangle<\Gamma$. Let $\psi:X_{\Gamma^{'}}\longrightarrow X_{\Gamma^{'}}/\langle Z\rangle$ denote the canonical projection.
\begin{proposition}
\label{ramif}
The branch locus of $\psi$ consists of three isolated points\linebreak $Q_1$,$Q_2$,$Q_3$. They are cyclic singularities of $X_{\Gamma}$, all of type $(7,3)$. Outside of $Q_1$,$Q_2$,$Q_3$, $X_{\Gamma}$ is smooth. The minimal resolution of each singularity $Q_i$, $i=1,2,3$, is a chain of three rational curves $E_{i,1}$, $E_{i,2}$, $E_{i,3}$ with selfintersections $(E_{i,1})^{2}=-3$, $(E_{i,2})^2=(E_{i,3})^2=-2$ and $(E_{i,1}\cdot E_{i,2})=(E_{i,2}\cdot E_{i,3})=1$, $(E_{i,1}\cdot E_{i,3})=0$ (Hirzebruch-Jung string of type $(-3)(-2)(-2)$). 
\end{proposition}
\noindent\emph{Proof.} The branch locus of $\psi$ doesn't depend explicitely on $\Gamma^{'}$ and is in fact the image of the fixed point set in $\BB$ of non-trivial finite order elements in $\Gamma$ under the canonical projection $\BB\longrightarrow \Gamma\backslash \BB$ coming from the ball. The number of its components is exactly the number of $\Gamma$-equivalence classes of elliptic fixed points in $\BB$. By (\ref{matrixrep}) the matrix representation $m_{\zeta}$ of $\zeta$ is just $m_{\zeta}=diag(\zeta,\zeta^2,\zeta^4)$. Only one (projectivized) eigenvector of $m_{\zeta}$--namely $e_1$--lies in the ball defined by $c$ and represents an elliptic fixed point. Let $x:=g^{-1} e_1$ denote the corresponding fixed point in $\BB$ of $Z$. Note that $\zeta$ can be embedded into $D$ in three different ways, namely as $\zeta$, $\zeta^{\sigma}$ or $\zeta^{\sigma\sigma}$. The two non-trivial embeddings give two further $\Gamma$-inequivalent fixed points $x^{\sigma}$ and $x^{\sigma\sigma}$ in the same way as $x$ is given. Let $Q_i\in X_{\Gamma}$, $i=1,2,3$ be the images of $x$, $x^{\sigma}$, $x^{\sigma\sigma}$ under the canonical projection. They give the three branch points. It is left to show that there are no more such points and that there are no curves in the branch locus. We will give an argument for it subsequent to the next proposition. Looking at the action of $\langle m_{\zeta}\rangle$ around $e_1$ we find that around $Q_i$ $X_{\Gamma_{\OO,b}}$ looks like $\CC^2/G$, with $G\cong \langle diag(\zeta,\zeta^3)$, which represents a cyclic singularity of type $(7,3)$. By standard methods we get the minimal resolution stated above.\\

Let $\widetilde{X_{\Gamma}}\stackrel{\rho}{\longrightarrow} X_{\Gamma}$ denote the minimal resolution of all singularities of $X_{\Gamma}$. Our goal is to determine the structure of $\widetilde{X_{\Gamma}}$. We start with topological invariants.
\begin{proposition}
\label{top}
$c_2(\widetilde{X_{\Gamma}})=12$, $\sign(\widetilde{X_{\Gamma}})=-8$. Consequently $c^2_1(\widetilde{X_{\Gamma}})=0$, $\chi(\widetilde{X_{\Gamma}})=1$.
\end{proposition}
\begin{proof}
In \cite{Ho:bsa}, R.~-P.~Holzapfel introduced two rational invariants of a two-dimensional complex orbifold $(X,B)$ (in the sense of \cite{Ho:bsa}), called the Euler height $\e(X,B)$ (see \cite{Ho:bsa}, 3.~3.) and the Signature height $\ttt(X,B)$ (\cite{Ho:bsa}, 3.~4.), which in the case of a smooth surface are the usual Euler number and the signature. In the general case, Euler- and Signature height contain contributions coming from the orbital cycle $B$, a marked cycle of $X$, which should be thought as a virtual branch locus of a finite covering of $X$. Most important result on these invariants is the nice property that they behave multiplicatively under finite coverings with respect to the degree. In particular, if $Y \stackrel{f}{\longrightarrow} (X,B)$ is a uniformization of $(X,B)$, i.~e.~a smooth surface which is a finite Galois cover of $(X,B)$, ramified exactly over $B$, then $c_2(Y)=deg(f)\e(X,B)$, $\sign(Y)=deg(f)\ttt(X,B)$. In our case, $X_{\Gamma^{'}}$ is an uniformization of the orbifold $(X_{\Gamma},Q_1,Q_2,Q_3)$. Since $X_{\Gamma^{'}}$ is a fake projective plane, we have $e(X_{\Gamma^{'}})=3$, $\sign(X_{\Gamma^{'}})=1$. Applying Holzapfels formulas \cite{Ho:bsa} prop.~3.3.4, and prop.~3.4.3, we get $e(X_{\Gamma})=3$, $\sign(X_{\Gamma})=1$. The birational resolution map $\rho$ consists of 9  monoidal transformations. Then, using \cite{Ho:bsa}, p.~142 ff, we obtain $e(\widetilde{X_{\Gamma}})=3+9$, $sign(\widetilde{X_{\Gamma}})=1-9$. The other invariants are immediately obtained using facts from the general theory mentioned at the end of section \ref{inv}.  
\end{proof}

\begin{proof}[Proof of propostion \ref{ramif} continued] From the proof of the above proposition we can deduce that there are no more branch points then we have found. Namely, if we assume that there are more, and knowing that no branch curves exist, we immediately obtain a contradiction to the equality between the orbital invariants $c_2(X_{\Gamma^{'}})=7\e(X_{\Gamma})=7(e(X_{\Gamma})-\sum (1-1/d_i))$ (by definition we have $\e(X_{\Gamma})=e(X_{\Gamma})-\sum (1-1/d_i)$, where sum is taken over the branch locus, and $d_i$ appears in the type $(d_i,e_i)$ of the branch point $Q_i$, see \cite{Ho:bsa}, 3.~3). Let us give an argument that no branch curves are possible. Such a curve must be subball quotient $C=\DD/G$, with $\DD\subset \BB$, $\DD\cong \BB_1$ a disc fixed by a reflection in $\Gamma$ and $G\subset \Gamma$ an arithmetic subgroup consisting of all elements in $\Gamma$ acting on $\DD$. Then $G$ is commensurable to a group of elements with reduced norm 1 in an order of a quaternion subalgebra $Q\subset D$, which is necessarily a division algebra. But for dimension reasons $D$ cannot contain quaternion algebras. Therefore $C$ doesn't exist. 
\end{proof}
In the next step we compute the irregularity and the geometric genus.
\begin{proposition}
\label{ireggeom}
$q(\widetilde{X_{\Gamma}})=p_g(\widetilde{X_{\Gamma}})=0$.
\end{proposition}
\begin{proof}
Due to the fact that $\chi(\widetilde{X_{\Gamma}})=1-q(\widetilde{X_{\Gamma}})+p_g(\widetilde{X_{\Gamma}})=1$, by preceding propostion \ref{top}, it suffices to show that one of the above invariants vanishes, let's say $p_g(\widetilde{X_{\Gamma}})=\dim H^0(\widetilde{X_{\Gamma}},\Omega^2_{\widetilde{X_{\Gamma}}})$. We know that $p_g(X_{\Gamma^{'}})=0$. Let $\Omega^2_{X_{\Gamma}}$ denote the space of holomorphic 2-forms on (the singular surface) $X_{\Gamma}$. Then $\Omega^2_{X_{\Gamma}}$ is exactly the space of $\langle \zeta\rangle$-invariant 2-forms on $X_{\Gamma^{'}}$, i.~e.~$\Omega^2_{X_{\Gamma}}=(\Omega^2_{X_{\Gamma^{'}}})^{\langle\zeta\rangle}$ (see \cite{Gri}). On the other hand we have an isomorphism between $\Omega^2_{\widetilde{X_{\Gamma}}}$ and $\Omega^2_{X_{\Gamma}}$ (again by \cite{Gri}). Altogether, $p_g(\widetilde{X_{\Gamma}})=0$. 
\end{proof}
Let us remark at this stage, that even if we know some invariants of $\widetilde{X_{\Gamma}}$, we still need to determine the Kodaira dimension in order to classify $\widetilde{X_{\Gamma}}$, since there exist surfaces with these invariants in every Kodaira dimension. We determine $\kod(\widetilde{X_{\Gamma}})$ discussing the first plurigenera of $\widetilde{X_{\Gamma}}$. Using an argument of Ishida \cite{Ishida} we first prove 
\begin{lemma}
\label{dim}
Let $A_k(\Gamma,j)$ denote the space of $\Gamma$-automorphic forms of weight $k$ with respect to the Jacobian determinant as the factor of automorphy and let $P_k(\widetilde{X_{\Gamma}})$  be the $k$-th plurigenus of $\widetilde{X_{\Gamma}}$. Then for $k=2,3$ $P_k(\widetilde{X_{\Gamma}})=\dim A_k(\Gamma,j)$. 
\end{lemma}
\begin{proof}
We can identify $A_k(\Gamma,j)$ with the space of $\langle\zeta\rangle$-invariant sections\\ $H^0(X_{\Gamma^{'}},\mathcal K_{X_{\Gamma^{'}}}^{\otimes k})^{\langle\zeta\rangle}$. Every such section can be regarded as a holomorphic section $s\in H^0(X^{sm}_{\Gamma},\mathcal K_{X^{sm}_{\Gamma}}^{\otimes k})$, where $X^{sm}_{\Gamma}=X_{\Gamma}\backslash\{Q_1,Q_2,Q_3\}$ denotes the smooth part of $X_{\Gamma}$. We can think of $X^{sm}_{\Gamma}$ as an open dense subset of $\widetilde{X_{\Gamma}}$. The crucial point is to show that $s$ has a holomorphic continuation along the exceptional locus. For this, let $(s)$ be the divisor of $\widetilde{X_{\Gamma}}$ corresponding to $s$ and write $(s)$ in three different ways as $(s)=a_{i,1}E_{i,1}+a_{i,2}E_{i,2}+a_{i,3}E_{i,3}+D_i$, $i=1,2,3$, with $E_{i,j}$ as in proposition \ref{ramif} and $D_i$ a divisor disjoint to $E_{i,j}$. Then, we have to show that $a_{i,j}$ are positive if $k=2,3$ is assumed. Let $K$ denote the canonical divisor of $\widetilde{X_{\Gamma}}$. We notice that $(s)$ and $kK$ are linearly equivalent. With our convention stated in proposition \ref{ramif} the adjunction formula gives the following intersection numbers: 
\begin{align}
\label{r1}
((s)\cdot E_{i,1})&=(kK\cdot E_{i,1})=k,\notag\\ 
((s)\cdot E_{i,2})&=(kK\cdot E_{i,2})=0,\\
((s)\cdot E_{i,3})&=(kK\cdot E_{i,3})=0.\notag
\end{align}
On the other hand, 
\begin{align}
\label{r2}
((s)\cdot E_{i,1})&=(a_{i,1}E_{i,1}+a_{i,2}E_{i,2}+a_{i,3}E_{i,3}+D_i\cdot E_{i,1})\notag\\
&=-3a_{i,1}+a_{i,2}+d_{i,1},\notag\\ ((s)\cdot E_{i,2})&=a_{i,1}-2a_{i,2}+a_{i,3}+d_{i,2},\\
((s)\cdot E_{i,3})&=a_{i,2}-2a_{i,3}+d_{i,3},\notag
\end{align}
with some integers $a_{i,j}$. Now, (\ref{r1}) and (\ref{r2}) lead to a system of linear equations, which in the case $k=2,3$ has positive solutions $a_{i,j}$, $j=1,2,3$. 
\end{proof}
In \cite{Hir:ell}, F.~Hirzebruch developed a formula for the dimension of spaces of automorphic forms $A_k(\Delta,j)$ with respect to a discrete and cocompact group which acts properly discontinously on some bounded hermitian symmetric domain with emphasis on ball quotient case. Let us recall this formula in the case of quotients of the n-dimensional ball:\\
Let $\Delta$ be a discrete group which acts properly discontinuously on the\linebreak n-dimensional ball $\BB_n$ with a compact fundamental domain. For $\delta\in\Delta$ let $\Delta_{\delta}$ be the centralizer of $\delta$ in $\Delta$, $Fix(\delta)$ the fixed point set of $\delta$ in $\BB_n$, and $m(\delta)$ the number of elements in $\Delta_{\delta}$ which act trivially on $Fix(\delta)$. If $r(\delta)$ denotes the dimension of $Fix(\delta)$ let $R(r(\delta),k)$ be the coefficient of $z^{r(\delta)}$ in the formal power series expansion of $(1-z)^{k(n+1)-1}\prod_{i=r(\delta)+1}^{n}\frac{1}{1-\nu_i+\nu_i z}$, where $\nu_{r(\delta)+1},\ldots,\nu_n$ are the eigenvalues of $\delta$ normal to $Fix(\delta)$. $R(r(\delta),k)$ is a polynomial in k of degree $r(\delta)$. Hirzebruch's result is:
\begin{equation}
\label{hirform}
\dim A_k(\Delta,j)=\sum_{[\delta]}\frac{e(\Delta_{\delta}\backslash Fix(\delta))j_{\delta}^k}{m(\delta)(r(\delta)+1)}R(r(\delta),k), 
\end{equation}
where $e(\Delta_{\delta}\backslash Fix(\delta))$ is the ``virtual Euler number`` (in the sense of \cite{Hir:ell}), $j_{\delta}$ is the Jacobian determinant evaluated at an arbitrary point of $Fix(\delta)$ and the sum is running over all conjugacy classes $[\delta]$ of elements with fixed points in $\BB_n$.\\
We apply this formula to the group $\Gamma$, which after some elementary calculations in combination with lemma \ref{dim} gives the following result.
\begin{proposition}
\label{pluri}
$P_2(\widetilde{X_{\Gamma}})=1$, $P_3(\widetilde{X_{\Gamma}})=4$.
\end{proposition}
As an immediate consequence we obtain
\begin{corollary}
\label{kod}
$\kod(\widetilde{X_{\Gamma}})=1$. Moreover, $\widetilde{X_{\Gamma}}$ is a minimal elliptic surface.
\end{corollary}
\begin{proof}
If $\widetilde{X_{\Gamma}}$ were of general type, the Riemann-Roch theorem would imply $P_2(\widetilde{X_{\Gamma}})\geq 2$, which contradicts the proposition \ref{pluri}. Also $\widetilde{X_{\Gamma}}$ is not rational by Castelnuovo's criterion. And lastly, if $\widetilde{X_{\Gamma}}$ were of Kodaira dimension 0, none of the plurigenera could be greater than one, which again gives a contradiction to proposition \ref{pluri}. $\widetilde{X_{\Gamma}}$ is minimal, since $c_2(\widetilde{X_{\Gamma}})=12$ (see \cite{bhpv}).
\end{proof}

\section{Another elliptic surface}
In this section we will study the ball quotient by an arithmetic group which contains $\Gamma$. Its desingularization turns out to be another elliptic surface, which has been already studied by Ishida \cite{Ishida} p-adically. From there we obtain the elliptic fibration on both of these surfaces. 
\subsection{Passage to the normalizer}
In general, the normalizer $N\Lambda$ in $\G(\mathbb R)$ of a principal arithmetic group $\Lambda\subset \G(\QQ)$ is a maximal arithmetic group (\cite{prasbo}, prop.~1.~4.). In fact, for the principal group $\Gamma=\Gamma_{\OO,b}$, we infer from \cite{PS}, 5.~4.~that the normalizer $\tilde\Gamma$ of $\Gamma$ in the projective group $\mathbb P\G(\mathbb R)=\G(\mathbb R)/\{\textrm{center}$\} contains $\Gamma$ with index 3. Morover, $\tilde\Gamma\cap\G_b(\QQ)=\Gamma$. It is easily shown, that the matrix $\tau=\left (\begin{smallmatrix} 0\ 0\ \alpha\\ 1\ 0\ 0\\0\ 1\ 0\end{smallmatrix}\right)$ has the order three, normalizes $\Gamma$, and lastly represents a class in $\mathbb P\G_b(\mathbb R)\cong\textrm{PU}(H_b)$. Consequently, $X_{\tilde\Gamma}=X_{\Gamma}/\langle \tau\rangle$. Let $X_{\Gamma}\stackrel{\varphi}{\longrightarrow} X_{\tilde\Gamma}$ denote the canonical projection. In the same way as in the lemma \ref{ramif}, we obtain the following
\begin{lemma}
\label{singng}
The ball quotient $X_{\tilde\Gamma}$ is smooth outside four points $Q,P_1,P_2,P_3$, which are cyclic quotient singularities of type $(7,3)$ (represented by $Q$) and $(3,2)$ (represented by $P_1,P_2,P_3$). The minimal resolution of $Q$ is a Hirzebruch-Jung string $A_1+A_2+A_3$ of type $(-3)(-2)(-2)$ and each of $P_i$-s is resolved by a Hirzebruch-Jung string $F_{i,1}+F_{i,2}$ of type $(-2)$ $(-2)$, $i=1,2,3$. 
\end{lemma}
\begin{proof}
Let $g\in D$ be the element introduced at the begining of section \ref{struct}. Using the relation $\tau g=g^{\sigma}\tau$, which in fact holds for any $g\in D$, it is directly checked that $\tau$ permutes the three lines $x,\ x^{\sigma}$ and $x^{\sigma\sigma}$ which are fixed by $Z$. Consequently the three singular points $Q_1,Q_2,Q_3$ of $X_{\Gamma}$ are mapped to one single point $Q\in X_{\tilde\Gamma}$ by $\varphi$. This point remains a quotient singularity of type $(7,3)$. There is at least one singularity more, call it $P_1$, coming from the positive definite eigenline of $\tau$ corresponding to the eigenvalue $\omega=\frac{1}{2}(-1+\sqrt{-3})$. It is a quotient singularity of type $(3,2)$ since around it $\tau$ acts as $diag(\omega,\omega^2)$. In order to show that there are two more singularities, we make use of the relation between orbifold invariants of $X_{\Gamma}$ and $X_{\tilde\Gamma}$. We know namely that $\e(X_{\Gamma})=3\e(X_{\tilde\Gamma})$. Furthermore, the (topological) Euler number $e(X_{\tilde\Gamma})$ equals $3$. In the same way as in the proof of proposition \ref{ramif} we exclude branch curves. Then by definition $\e(X_{\tilde\Gamma})=3-6/7-\sum_{k=1}^{r} (1-1/d_k)$, where $r$ denotes the number of elliptic branch points $\neq Q$, and $d_k$ appears in the type $(d_k,e_k)$ of the k-th branch point. On the other hand $\e(X_{\Gamma})=3/7=3\e(X_{\tilde\Gamma})$. This holds only if $r=3$ and $d_k=3$ for all $k$, as a short calculation shows. This gives two further branch points $P_2$ and $P_3$. Doing the same type of argumentation with the signature height we conclude that all branch points, not not of type $(7,3)$ must be of type $(3,2)$. Namely, assuming the opposite we always get a contradiction to the equation $\ttt(X_\Gamma)=3\ttt(X_{\tilde\Gamma})$. 
\end{proof}

Again, we can ask about the structure of the minimal desingularization $\widetilde{X_{\tilde\Gamma}}$ of $X_{\tilde\Gamma}$ as we did before for $X_{\Gamma}$. With the same methods used in the investigation of $X_{\Gamma}$ we get 
\begin{proposition}
\label{NG}
$\widetilde{X_{\tilde\Gamma}}$ is a minimal elliptic surface of Kodaira dimension one with $p_g=q=0$.  
\end{proposition}
\begin{proof}
The topological invariants are computed using the Euler- and Signature height presented in the proof of proposition \ref{top} and lemma \ref{singng}. We get the same topological invariants as in proposition \ref{top}: $e(\widetilde{X_{\tilde\Gamma}})=12$, $\tau(\widetilde{X_{\tilde\Gamma}})=-8$. The assertion about the irregularity and the geometric genus follows directly from the (proof of) proposition \ref{ireggeom}, since we have $\chi(\widetilde{X_{\tilde\Gamma}})=1$ again. Lastly, we can apply Hirzebruch's formula in order to compute the second and third plurigenus, since the proof of lemma \ref{dim} works in the present case without any change. Therefore, we can identify these plurigenera with the dimensions $\dim A_k(\tilde\Gamma,j)$ of the corresponding spaces of automorphic forms. By elementary calculations, (\ref{hirform}) leads to $P_2(\widetilde{X_{\tilde\Gamma}})= P_3(\widetilde{X_{\tilde\Gamma}})=1$. A slightly modified argumentation in the proof of corollary \ref{kod} verifies the asserted Kodaira dimension (notice that $P_3=1$ and $p_g=q=0$ is not possible in Kodaira dimension 0).
\end{proof}

\subsection{Elliptic fibration}
We have to mention, that alternatively to the approach we have described, for the proof of proposition \ref{NG} we can completely refer to \cite{Ishida}, some of whose arguments we have already used before. There, the author a priori works over a non-archimedian field, but most of his arguments work independently of it. Morover, in \cite{Ishida}, section 4, the singular fibers of the elliptic fibration on $X_{\tilde\Gamma}$ are completely determined. The non-multiple singular fibers are closely related to the exceptional curves on $\widetilde{X_{\tilde\Gamma}}$.To be precise, we have 
\begin{theorem}[compare \cite{Ishida}, section 4]
$\widetilde{X_{\tilde\Gamma}}$ admits an elliptic fibration $f$ over $\PP_1$. $f$ has exactly one multiple fiber of multiplicity 2 and one multiple fiber of multiplicity 3. Furthermore, it has four non-multiple singular fibers, all of type $I_3$ (in Kodaira's notation) $B_0=A_2+A_3+D_0$, $B_1=F_{1,1}+F_{1,2}+D_1$, $B_2=F_{2,1}+F_{2,2}+D_2$, $B_3=F_{3,1}+F_{3,2}+D_3$. 
\end{theorem}

We can now use the knowledge of the elliptic fibration on $X_{\tilde\Gamma}$ to reconstruct the elliptic fibration on $X_{\Gamma}$. Since we know the finite covering $\varphi$, this is not a difficulty anymore. Again the non-multiple singular fibers contain the exceptional curves. For the proof of the next theorem we can also refer to \cite{keum} whose starting point was exactly the determination of the elliptic fibration.
\begin{theorem}[see \cite{keum}, proposition 2.~1.]
The elliptic fibration $g$ on $X_{\Gamma}$ over $\PP_1$ has exactly two multiple fibers, one of multiplicity two and one of multiplicity three. It has four non-multiple singular fibers, one of type $I_9$: $C_0=E_{1,2}+E_{1,3}+E_{2,2}+E_{2,3}+E_{3,2}+E_{3,3}+D_{1,0}+D_{2,0}+D_{3,0}$, and three of type $I_1$: $A_i=D^{'}_i$, $i=1,2,3$. There $D^{'}_i$ is the inverse image of $D_i$ under $\varphi$.
\end{theorem}

\bibliography{literatura}

\begin{thebibliography}{BHPVdV04}

\bibitem[BHPVdV04]{bhpv}
Wolf~P. Barth, Klaus Hulek, Chris A.~M. Peters, and Antonius Van~de Ven.
\newblock {\em Compact complex surfaces}.
\newblock Ergebnisse der Mathematik und ihrer Grenzgebiete. 3.~Folge, Bd.~4.
  Springer-Verlag, Berlin, 2.~ed edition, 2004.

\bibitem[BP89]{prasbo}
Armand Borel and Gopal Prasad.
\newblock Finiteness theorems for discrete subgroups of bounded covolume in
  semi-simple groups.
\newblock {\em Inst. Hautes \'Etudes Sci. Publ. Math.}, (69):119--171, 1989.

\bibitem[BR00]{br}
Don Blasius and Jonathan Rogawski.
\newblock Cohomology of congruence subgroups of {${\rm SU}(2,1)\sp p$} and
  {H}odge cycles on some special complex hyperbolic surfaces.
\newblock In {\em Regulators in analysis, geometry and number theory}, volume
  171 of {\em Progr. Math.}, pages 1--15. Birkh\"auser Boston, Boston, MA,
  2000.

\bibitem[Gri76]{Gri}
Phillip~A. Griffiths.
\newblock Variations on a theorem of {A}bel.
\newblock {\em Invent. Math.}, 35:321--390, 1976.

\bibitem[Hir66]{Hir:ell}
Friedrich Hirzebruch.
\newblock Elliptische {D}ifferentialoperatoren auf {M}annigfaltigkeiten.
\newblock In {\em Festschr. Ged\"achtnisfeier K. Weierstrass}, pages 583--608.
  Westdeutscher Verlag, Cologne, 1966.

\bibitem[Hol98]{Ho:bsa}
Rolf-Peter Holzapfel.
\newblock {\em Ball and surface arithmetics}.
\newblock Aspects of Mathematics, E29. Friedr. Vieweg \& Sohn, Braunschweig,
  1998.

\bibitem[IK98]{Ishida-Kato}
Masa-Nori Ishida and Fumiharu Kato.
\newblock The strong rigidity theorem for non-{A}rchimedean uniformization.
\newblock {\em Tohoku Math. J. (2)}, 50(4):537--555, 1998.

\bibitem[Ish88]{Ishida}
Masa-Nori Ishida.
\newblock An elliptic surface covered by {M}umford's fake projective plane.
\newblock {\em Tohoku Math. J. (2)}, 40(3):367--396, 1988.

\bibitem[Kat08]{Kat}
Fumiharu Kato.
\newblock On the {S}himura variety having {M}umford's fake projective plane as
  a connected component.
\newblock {\em Math. Z.}, 259(3):631--641, 2008.

\bibitem[Keu06]{keum}
JongHae Keum.
\newblock A fake projective plane with an order 7 automorphism.
\newblock {\em Topology}, 45(5):919--927, 2006.

\bibitem[Mum79]{mum2}
David Mumford.
\newblock An algebraic surface with {$K$} ample, {$(K\sp{2})=9$},
  {$p\sb{g}=q=0$}.
\newblock {\em Amer. J. Math.}, 101(1):233--244, 1979.

\bibitem[Pra89]{pras}
Gopal Prasad.
\newblock Volumes of {$S$}-arithmetic quotients of semi-simple groups.
\newblock {\em Inst. Hautes \'Etudes Sci. Publ. Math.}, (69):91--117, 1989.

\bibitem[PY07]{PS}
Gopal Prasad and Sai-Kee Yeung.
\newblock Fake projective planes.
\newblock {\em Invent. Math.}, 168(2):321--370, 2007.

\bibitem[Rei03]{reiner}
Irving Reiner.
\newblock {\em Maximal orders}.
\newblock London Mathematical Society Monographs. New Series, 28. The Clarendon
  Press Oxford University Press, Oxford, 2003.

\bibitem[Rie70]{Rie}
Carl Riehm.
\newblock The norm 1 group of a $\mathfrak p$-adic division algebra.
\newblock {\em Amer. J. Math.}, 92:499--523, 1970.

\bibitem[Tit79]{Tits}
Jacques Tits.
\newblock Reductive groups over local fields.
\newblock In {\em Automorphic forms, representations and $L$-functions (Proc.
  Sympos. Pure Math., Oregon State Univ., Corvallis, Ore., 1977), Part 1},
  Proc. Sympos. Pure Math., XXXIII, pages 29--69. Amer. Math. Soc., Providence,
  R.I., 1979.

\end{thebibliography}
\bibliographystyle{alpha}
\end{document}